\documentclass[a4paper,11pt]{article}  

\usepackage{graphicx,color,subfigure,multirow, here}
\usepackage{mathtools}
\usepackage[utf8]{inputenc} 
\usepackage[T1]{fontenc}
\usepackage[english]{babel}
\usepackage{amsmath, amsfonts, amsthm,amssymb,here,dsfont, hyperref}
\usepackage[numbers]{natbib}

\newcommand{\one}{\mathds{1}}

\newtheorem{theo}{Theorem}[section]
\newtheorem{defi}[theo]{Definition}
\newtheorem{prop}[theo]{Proposition}

\newtheorem{rem}[theo]{Remark}
\newtheorem{ass}[theo]{Assumption}

\newcommand{\n}{\noindent }

\newcommand{\E}{\mathbb{E}}
\newcommand{\R}{\mathbb{R}}
\newcommand{\N}{\mathbb{N}}

\renewcommand{\P}{\mathbb{P}}

\makeatletter

\@addtoreset{equation}{section}
\makeatother

\begin{document}

\title{Exponential ergodicity for diffusions \\with jumps driven by a Hawkes process}

\maketitle

\begin{center}
\author{Charlotte Dion$^{(1)}$ and Sarah Lemler$^{(2)}$ and Eva L\"ocherbach$^{(3)}$ \vspace{3mm}\\

$^{(1)}$ Sorbonne Université, UMR CNRS 8001, LPSM, 75005 Paris, France
\\
$^{(2)}$ \'Ecole CentraleSup\'elec, MICS, 91190 Gif-sur-Yvette, France
\\
$^{(3)}$ Université Paris 1, SAMM EA 4543, 75013 Paris, France}
\end{center}

\section*{Abstract}

In this paper, we introduce a new class of processes which are diffusions with jumps driven by a multivariate nonlinear Hawkes process. Our goal is to study their long-time behavior. In the case of exponential memory kernels for the underlying Hawkes process we establish conditions for the positive Harris recurrence of the couple $(X,Y )$, where $X$ denotes the diffusion process and $Y$ the piecewise deterministic Markov process (PDMP) defining the stochastic intensity of the driving Hawkes. As a direct consequence of the Harris recurrence, we obtain the ergodic theorem for $X.$  Furthermore, we provide sufficient conditions under which the process is exponentially $\beta-$mixing.

\n {\bf AMS Classification}: 60J35, 60F99, 60H99, 60G55 \\

\n {\it Keywords:} Diffusions with jumps; nonlinear Hawkes process, Piecewise deterministic Markov processes (PDMP), Ergodicity; Foster-Lyapunov drift criteria; $\beta$-mixing.

\section{Introduction}\label{sec:intro}

\subsection{Motivation}


In this work, we introduce a new class of continuous time processes $ X = (X_t)_{t \geq 0}$ taking values in $ \R,$  with jumps driven by a multivariate nonlinear Hawkes process $N, $ satisfying the following equation 
\begin{equation}\label{eq:1}
X_t=X_0+\int_0^t b(X_s)ds+\int_0^t\sigma(X_s)dW_s+\int_0^ta(X_{s-})\sum_{j=1}^MdN_s^{(j)},
\end{equation}
where $W$ is the standard one-dimensional Brownian motion and $N$ is a multivariate Hawkes process having $M$ components with intensity process $\lambda.$ $N$ is supposed to be independent from $W.$
We shall specify the dynamics of $N$ later on, in particular, the precise form of the stochastic intensity process $ \lambda $ is given in \eqref{eq:intensity} below. 

The jumps of the Hawkes process impact the dynamic of the diffusion process. In this model, the structure of the jumps is different from classical jump-diffusion processes with L\'evy jumps. Indeed: \\
1) the intensity of the jump process $N$ does not depend on the dynamic of $X$,\\
2) the Hawkes process has a special structure of time dependency: the state of the intensity at time $t$ depends on the entire past,\\
3) the multidimensional nature of the Hawkes process can be interpreted as the influence of $M$ subjects communicating one with each other and impacting the dynamic of $X$ along time. \\

Our main motivation is to be able to describe systems of interacting neurons, where we are in particular interested in modeling 
 the membrane potential of a single neuron together with the sequence of spike times of its presynaptic neurons. Indeed, neurons are communicating through the emission of electrical signals, forming a network. 
When a (presynaptic) neuron sends a signal to another (postsynaptic) neuron, the membrane potential of the  neuron increases sharply, and we say that the neuron spikes at this time.  If we are looking at one fixed neuron at the heart of this network, its activity between two {spikes} of one of its presynaptic neurons can be designed by the present model. Indeed, during this interval the membrane potential can be described by a diffusion process influenced by the spikes coming from the presynaptic neurons which are modeled by the Hawkes process. 
Hawkes processes have been studied recently in the context of neurosciences \citep[see][]{RBRTM, hodara2017} to model the interactions between the neurons within their sequences of spikes. The new model we propose allows to use both informations: the continuous membrane potential for a fixed neuron together with the spikes of several other neurons around it.

The aim of the present article is to study the longtime behavior of the process $X.$ Since $X$ is not Markov on its own, we propose a study of the couple of processes $(X, \lambda )$  in the case when the intensity $ \lambda $ of $N$ can be described in terms of a piecewise deterministic Markov process (PDMP) $Y,$ see \eqref{eq:aux} below. Such a description is possible  when the memory kernels are {\it exponential}. We prove the positive Harris recurrence of $(X,Y)$ together with ergodicity, and we establish the speed of $\beta-$mixing. These are important probabilistic properties which are interesting in its own right but which are also the key ingredients for the theoretical study of estimators of the model parameters. The estimation of the coefficients of the process is the issue of the companion paper \cite{charlottesarah}.

\subsection{State of the art}
Time-homogeneous diffusion processes are well known Markov processes that have been widely studied. Let us cite for example \cite{itodiffusion} which has initiated the work on Markovian properties of diffusion processes or the by now classical references \cite{ikeda2014stochastic,revuz} in which all properties of diffusions are summarized. 
Ergodicity and strong mixing have been established in the pioneering papers \cite{veret, veret2}. 
Then, jump-diffusion processes have been introduced, for example in the context of risk management \citep[see][]{TANKOV}, either driven by a Poisson process or more generally by a L\'evy process \citep[see][for a review]{applebaum}. Their ergodicity and mixing properties have been studied from a probabilistic point of view by \cite{MASUDA}.

On the other hand, Hawkes processes \citep[][]{HAWKES71} have been studied a lot lately. 
These  point processes generalize the Poisson processes by introducing what is commonly called ``self-excitation'': Present jumps are able to trigger (or to inhibit) future jumps, and in this way, correlations between successive jump events are introduced. 
\cite{DVJ} gives a precise synthesis of earlier
results published by Hawkes.
Important probabilistic results for these mutually exciting point processes, such as stability, limit theorems and large population limits, have been obtained in \cite{BM1996, DFH, BDH, BACRY, Clinet, MADS}. Ergodicity for general non-linear Hawkes processes with general memory kernels has been recently established by \cite{MADS}. For linear Hawkes processes, \cite{Carl} proves exponential ergodicity, giving very precise estimates on the regeneration times. Finally \cite{Clinet} establishes exponential ergodicity in the particular case of linear Hawkes processes with exponential memory kernels. 

Relying on one of these recent results on ergodicity of Hawkes processes, it is therefore reasonable to argue that the solution of \eqref{eq:1} should be ergodic itself, provided we are able to ensure that the diffusion part of \eqref{eq:1}, that is the dynamic \eqref{eq:1} with $ a (\cdot ) \equiv 0,$ driving the motion in between successive jumps, is ergodic.

However, no such general criterion,  ensuring that the ergodicity of the underlying jump process implies the ergodicity  of the associated diffusion with jumps, exists, at least not to our knowledge. 

The present paper proposes therefore a short study of the ergodicity of diffusions with jumps driven by Hawkes processes with exponential memory kernels. In this case, the associated stochastic intensity processes of the jump part can be expressed in terms of particular Markov processes, PDMP's, see e.g. \cite{DLO16}. For these PDMP's, ergodicity has been established by \cite{Clinet} in the case of linear Hawkes processes, and by \cite{DLO16} in the case of non-linear Hawkes processes with Erlang kernels. Here, we extend these results to the general non-linear case, including also inhibitions \footnote{Notice that \cite{CHEN2017} has obtained a control of the $\tau$-mixing coefficient in the inhibitory case -- but this does not imply the $ \beta-$mixing property of the process. }, 
and more importantly, we prove the joint ergodicity as well as the  $\beta$-mixing property for the couple $(X,Y )$ in the case of exponential kernel functions. Finally, we give criteria ensuring that the mixing is exponential. 

\subsection{Main contributions and plan of the paper}
We propose a general study of the longtime behavior of a jump diffusion process, where the jumps are driven by a multivariate non-linear Hawkes process in dimension $M,$ having exponential memory kernels. 
Our proof of the Harris recurrence of $(X, Y)$ follows a well-known scheme, see e.g. \cite{MT1992, MT1993}. 

Firstly, we establish a local D\"oblin lower bound for the transition semigroup of the process $ (X, Y ) ,$ based on a lower bound of the joint transition density of $(X, Y).$  This is the content of Theorem \ref{thm:Doeblin}. The ideas that we use to obtain this lower bound are very natural, and they have appeared independently of our work already in  \cite{Clinet}. They have also been exploited in \cite{DLO16} and in \cite{evaSPA}: more precisely, we use $M^2$ consecutive jumps which are carefully chosen to produce a part of Lebesgue density for the transition of $Y.$  Once the jump part of the process possesses some part of Lebesgue density, we have to ensure that this density is preserved by the stochastic flow. This part of the proof relies on fine support properties of $ (X,Y)$ and provides, as a by-product, a simulation algorithm for the process. In particular we need to localize the possible positions of the diffusion part, conditionally on the jump times of the underlying Hawkes process that have already been chosen. 

In a second step, a Foster-Lyapunov condition (see Proposition \ref{prop:lyapunov}) ensures the control of the return times to some compact set $K.$ We then use control arguments to deduce from this a precise control of the hitting times of the set $C$  where the transition densities are lower-bounded according to the local D\"oblin lower bound.  Both results together imply the positive Harris recurrence of the couple $(X, Y), $ as well as the ergodic theorem. This is our first main result, stated in Theorem \ref{harrisrec}. 

Finally, following the steps of \cite{MASUDA} where the ergodicity and exponential $\beta$-mixing bounds are established for diffusions with jumps driven by a L\'evy process and the general framework of \cite{kulik},  our second main theorem \ref{theo:second} states the mixing property of the process together with a control on the rate of convergence. To conclude, our Theorem \ref{thm:mixing} provides sufficient criteria implying the exponential  control of the $\beta-$mixing coefficient of $(X,Y ).$

Finally, let us mention that it is possible to extend our results to multivariate diffusions driven by Hawkes processes. However, the primary aim of our article is to provide a theoretical probabilistic foundation for a statistical application where we want to deal with  data describing the membrane potential of one single neuron together with the spike trains (the point process data associated to the subsequent times of action potentials) of its presynaptic neurons. This is why we concentrate on a one-dimensional diffusion model, driven by a multivariate Hawkes process in this paper. Indeed, in this modeling context, the joint evolution of the membrane potentials of, say, $d$ neurons, is independent, conditional on the evolution of the spike trains of their presynaptic neurons. This is due to the fact that the membrane potential processes of different neurons within a network do only interact via the incoming stimuli, that is, the incoming spike trains of the presynaptic neurons. In other words, their driving Brownian motions are independent - no interaction is present in their drift nor in their diffusion coefficient.

Our paper is organised as follows. We introduce our model together with all necessary notation and assumptions in Section \ref{sec:notations}. Section \ref{sec:useful} provides a simulation algorithm of the process which is interesting in its own right. In Section \ref{sec:proba}, we prove the ergodicity and finally the exponential $\beta-$mixing of the coupled process $( X, Y).$

\section{Presentation of the model and first properties}\label{sec:notations}

\subsection{The model}
Throughout this paper we work on a filtered probability space $(\Omega,{\mathcal F} , \mathbb{F}).$
We start by introducing the jump part of our process,  that is, the multivariate Hawkes process. 

This process is defined in terms of a collection of jump rate functions  $  f_i : \R \to \R_+ , 1 \le i \le M ,$ and in terms of interaction functions (also called {\it memory kernels})  $ h_{i j } : \R _+ \to \R  , 1 \le i, j \le M ,$ which are measurable functions. Let moreover $n^{(i)} , 1 \le i \le M,$  be discrete point measures on $\R_{-} $ satisfying that 
$$ \int_{\R_{-} }| h_{ij}|   ( t- s) n^{(j)} (ds ) < \infty \mbox{ for all } t \geq 0 .$$
We interpret them as {\it initial condition} of our process. 

\begin{ass}\label{ass:0} 
We suppose that each $ f_i , 1 \le i \le M,$ is Lipschitz continuous with Lipschitz constant $ \gamma_i .$ 
\end{ass}

\begin{defi}
A non-linear Hawkes process with parameters $ (f_i, h_{ij})_{1 \le i,j  \le M} , $ and with initial condition $ n^{(i)}, 1 \le i \le M,$ is a multivariate counting process $N_t= (N_t^{(1)}, \ldots, N_t^{(M)} ) , t \geq 0,$ such that 
\begin{enumerate}
\item
$\P-$almost surely, for all $ i \neq j  ,  N^{(i)} $ and $ N^{(j) }  $ never jump simultaneously,
\item
for all $1 \le i \le M, $ the compensator of $N^{(i) }_t $ is given by $\int_0^t \lambda^{( i) }_s ds ,  $ where 
\begin{equation}\label{eq:intensity}
\lambda^{(i)}_t= f_i \left(  \sum_{j=1}^M \int_{0}^{t-} h_{i j }(t-u)dN^{(j) }_u  + \sum_{j=1}^M \int_{-\infty}^{0} h_{i j }(t-u)dn^{(j) }_u \right) .
\end{equation}
\end{enumerate}
\end{defi}
If the functions $ h_{ij } $ are locally integrable, the existence of the process $(N_t)_{t \geq 0}$ with the prescribed intensity \eqref{eq:intensity}  on finite time intervals follows from standard arguments, see e.g. \cite{BM1996} or \citep{DFH}.

We consider a jump diffusion process $X = (X_t)_{ t \geq 0 } $ taking values in $\R, $ whose jumps are governed by $N.$ More precisely, $X$ is solution of the stochastic differential equation 
\begin{equation}\label{model}
X_t =X_0 +\int_0^t b(X_s) ds + \int_0^t  \sigma(X_s)dW_s + \int_0^t a(X_{s -}) \sum_{j =1}^M  dN^{(j)}_s,
\end{equation}
where $X_0$ is a random variable independent of $W$ and of $ N .$ Here $W$ is the standard Brownian motion in dimension one, independent of the multivariate Hawkes process $N $ having intensity \eqref{eq:intensity}.

\begin{prop}
Suppose that Assumption \ref{ass:0} holds, that the coefficients $b, \sigma $ are globally Lipschitz continuous, that $a$ is measurable and that the memory kernel functions $h_{ij} $ are locally integrable. Then equation (\ref{model}) admits a unique strong solution.
\end{prop}

\begin{proof}
Theorem 6 of \cite{DFH} implies that for any $ T > 0,$  almost surely, $ N_t $ has only a finite number of jumps on  $[0, T ].$ We can therefore work conditionally with respect to the choice of these jumps and construct the solution of \eqref{model} by pasting together trajectories of strong solutions to the diffusion equation (that is, the solution of \eqref{model} with $ a( \cdot ) \equiv 0$)  at the successive jump times (see \cite{ikeda1966} for related ideas).
\end{proof}

In the sequel we shall also need estimates on the transition densities of the diffusion process underlying \eqref{model}. We want to be able to deal with a general drift term $b$ that is not bounded in order to include Ornstein-Uhlenbeck type processes to the class of models that we consider. Unbounded drift coefficients add however a substantial difficulty when one wishes to apply techniques from Malliavin calculus to obtain bounds on the transition densities. Therefore, we impose the following additional assumptions which are inspired by \cite{gobet}.  

\begin{ass}\label{ass:coeff}
\begin{enumerate}
\item $a$ is continuous.
\item $ b $ and $ \sigma $ are of class $\mathcal{C}^2. $
\item There exist positive constants $c, q $ such that for all $x \in \R, $ $  | b' ( x) | +| \sigma' (x) | \le c  $ and  $ | b'' ( x) | + | \sigma '' (x)| \le c ( 1 + |x|^q )  .$    
\item There exist strictly positive constants $\sigma_0 $ and $\sigma_1$ such that $ \sigma_0 \le  \sigma^2(x) \le \sigma_1$ for all $ x \in \R.$ 
\end{enumerate}
\end{ass}

\begin{rem}
We do only need the above assumption to obtain good lower bounds on the transition densities of the underlying non-jumping diffusion process, see \eqref{eq:aronson} below. 
\end{rem}

To study the longtime behavior of $X$ and to ensure its ergodicity, we introduce two additional conditions which are classical in the study of diffusion processes and which are derived from \cite{Has}, see also \cite{veret}. 

In the following, we assume we are in one of the following frames. 
\begin{ass}\label{ass:erg-expo}
{\bf Exponential frame.} There exist $d \geq 0$, $ r > 0 $ such that for all  $x$ satisfying  $ |x| >r$, we have $xb(x) \leq -d x^2.$ Moreover, one of the two following conditions holds true. 
\begin{enumerate}
\item
For all $ |x| > r, $ $ 2 x a(x) + a^2 (x) \le  0.$ 
\item
For all $ |x| > r, $ $ | a(x) | \le C |x|^\eta $ for some $ \eta < 1.$ 
\end{enumerate}
\end{ass}
\begin{ass}\label{ass:erg-poly}
{\bf Polynomial rate.} There exist $ \gamma > \sigma_1 / 2 , r > 0 $ and $ m \in ] 2, 1 + \frac{ 2 \gamma}{\sigma_1^2}[, $ such that for all $x$ satisfying $ |x| > r, $ we have $ x b(x) \leq - \gamma$  and  $ | x + a (x) |^m - |x|^m  \le  0.$
\end{ass}
\subsection{Markovian framework}
Our study relies on the general theory of Markov processes. To be able to work within a Markovian framework, we impose a special structure on the interaction functions and suppose that 
\begin{ass}\label{ass:markov}
For all $ 1 \le i, j \le M, $
\begin{equation}\label{eq:defh}
h_{ij } ( t) = c_{ij} e^{- \alpha_{ij} t }, c_{ij} \in \R , \alpha_{ij} > 0. 
\end{equation}
\end{ass}
In this case we may introduce the auxiliary Markov process $ Y = Y^{(ij)} ,$ 
\begin{equation}\label{eq:aux}
 Y_t^{(ij)} = c_{ij } \int_0^t e^{- \alpha_{ij} (t-s)} d N_s^{(j) } + c_{ij } \int_{-\infty}^0 e^{- \alpha_{ij} (t-s)} d n_s^{(j) }  , 1 \le i, j \le M .
\end{equation}
The intensities can be expressed in terms of sums of these Markov processes, that is, for all $ 1 \le i \le M,$ 
$$   \lambda_t^{(i ) } = f_i \left(  \sum_{j=1}^M Y^{(ij)}_{t-} \right) . $$
Notice that  
\begin{equation}\label{eq:gen}
Y_0^{(ij)} = y_0^{(ij)} =  c_{ij } \int_{-\infty}^0 e^{ \alpha_{ij} s} d n_s^{(j) } , \; \; \mbox{ and } \; \;  d Y_t^{(ij )} = - \alpha_{ij } Y_t^{(ij)} dt + c_{ij} d N_t^{(j)}.
\end{equation}
In particular, the process 
$$(X_t, Y_t^{(ij)}, 1 \le i, j \le M ) 
$$
is a Markov process. Its longtime behavior is determined on the one hand by the longtime behavior of the underlying Hawkes process $N_t$ and on the other hand on the one of the continuous diffusion process with drift $b$ and diffusion coefficient $\sigma.$ Note also that $N_t $ is an autonomous process, that is, $ N_t$ does not depend on $X.$

\subsection{First properties of the process $ (X, Y) $ and an associated stochastic flow.}\label{sec:useful}
In the sequel, we shall denote the whole process by $ Z_t := ( X_t, Y_t) .$ It takes values in $ \R \times \R^{M \times M}. $  We write $ (P_t)_{t \geq 0 } $ for its associated transition semigroup and $ z = ( x, y ) , y = (y^{(ij)})_{1 \le i, j \le M} $ for generic elements of its state space.  
\begin{prop}\label{prop:Feller}
Grant Assumptions \ref{ass:0} and  \ref{ass:coeff}. Then the process $ Z$ is a Feller process, that is, $ P_t \Phi  $ belongs to $ C_b ( \R \times \R^{M \times M}  ) $ whenever $\Phi  \in C_b ( \R \times \R^{M \times M}  ).$
\end{prop}

The proof of this result follows from classical arguments, see e.g.\ the proof of Proposition 4.8 in \cite{evaflow}, or \cite{ikeda1966}. 
In what follows we shall give an alternative proof relying on an explicit construction of the process $Z$ as a {\it stochastic flow}. This construction is interesting in its own right and based on the fact that $Z$ is a {\it piecewise continuous Markov process}, that is, a generalization of the piecewise deterministic Markov processes to those traveling in between successive jumps according to {\it stochastic flows} instead of deterministic ones, see \cite{ikeda1966}.  We start by giving the principal elements needed to construct this flow. \\

{\it The associated stochastic Brownian flow.} Thanks to our assumptions on $b$ and $ \sigma, $ by Theorem 4.2.5 of \cite{Kunita}, there exists a unique stochastic flow of homeomorphisms $ \Phi_{ t }(x)$ for $0  \le t < \infty , x \in \R $, such that
\begin{equation}\label{eq:flow}
  \Phi_{ t } (x) = x + \int_0^t b ( \Phi_{u } (x) ) du + \int_0^t \sigma ( \Phi_{ u }(x) ) d W_u .
\end{equation}
In particular, we have that 
$$
 (t,x) \mapsto \Phi_{t}(x) \mbox{  is continuous } (0 {\leq}t{<}{\infty}, ~x \in \R) . $$
The stochastic flow $ \Phi $ describes the evolution of $ X$ in between successive jumps of $X$ (or, equivalently, of $N$). 

{\it The associated deterministic flow.} In between successive jump events of $N, $ the process $Y$ evolves according to the deterministic flow 
\begin{equation}\label{eq:flowygrand}
 \varphi_t (y) =  ( \varphi^{(ij)}_t ( y))_{1 \le i,j \le M}  , \; \mbox{ where } \varphi_t^{(ij)} (y ) = e^{ - \alpha_{ij} t} y^{(ij)} , 1 \le i ,j \le M.
\end{equation}

We are now ready to give the simulation algorithm for $Z$.

{\it Simulation algorithm for $Z.$} 
We propose a simulation algorithm for $Z$ for any family of starting configurations $z \in K_1 \times K_2, $ where $ K_1 \subset \R , $ $K_2 \subset \R^{M \times M} $ are compact sets. The first step is to construct an upper bound $ \bar N_t $ on the number of jumps of $Z, $ that is, of $N,$ during some finite time interval $ [0, t ].$ To do so, observe that by Lipschitz continuity of $f_i$ with Lipschitz constant $\gamma_i , $ we have for $1 \leq i \leq M$,
$$ f_i ( x) \le f_i ( 0 ) + \gamma_i |x| .$$
Let
\begin{equation}\label{cbar}
\bar \gamma := \max_i \gamma_i, \; 
\bar c := \sup_{1 \le i, j \le M  } |c_{i j}|  .
\end{equation}

Consider now the one-dimensional auxiliary linear Hawkes process $N_t^* $ having intensity $\lambda_t^*, t \geq 0,  $ where  
$$ \lambda_0^* :=  \sup_{ y \in K_2}  \sup_{ t \geq 0} \sum_{ i=1}^Mf_i (\sum_{j=1}^M e^{- \alpha_{ij } t} y^{(ij)} ) $$
and 
$$  \lambda_t^* = \lambda_0^* +M \bar \gamma \bar c \int_0^{t-}  dN_u^*  =   \lambda_0^*  + M \bar \gamma \bar c N^*_{t-} , ~ t \geq 0 .$$
Observe that $ N^* $ is a one dimensional linear Hawkes process with memory kernel $ \bar h ( t) = M \bar \gamma \bar c \one_{ \R_+} ( t) .$ 
Obviously, for all $t \geq 0, $ 
$$ \sum_{i=1}^M N_t^{(i) } \le  N_t^* . $$ 
Moreover, for all $t \geq 0,  N_t^* < \infty $ almost surely. In what follows, we shall write $T_1^* <  T_2^* < \ldots <  T_n^* < \ldots $ for the successive jump times of $ N^*.$ 

We work conditionally on the realization of $ N^* $ on $ [0, t ], $ that is, on the event $ \{  N_t^* = n \} , $ and on a realization of the associated jump times $0 < t_1 < t_2 \ldots < t_n < t.$  Our goal is to construct a version of $ Z, $ conditionally on these choices, which is continuous in the starting point $z.$ This construction relies on the classical thinning method, also known as acceptance-rejection method. 

During this construction,  we choose successively random variables $ U_1, \ldots, U_n  $ taking values in $ \{ 0, 1 , \ldots, M \} $ and define a process $z_s =  z_s ( z, t_1^n , U_1^n ) ,$ depending on these choices, for $ 0 \le s \le t.$ Here, $ t_1^n = (t_1, \ldots, t_n) $ denotes the choices of the successive jump times and $U_1^n = (U_1, \ldots , U_n ) $ the sequence of indices of jumping particles. This process is defined recursively as follows. Firstly, we put 
$$ z_s = (x_s, y_s) =   (\Phi_s ( x ) , \varphi_s ( y ) )  \mbox{ for all } 0 \le s < t_1 .$$
Then, conditionally on $y_{t_1 -}  = y_1 , $ we choose a random variable $ U_1 \in \{ 0 , 1, \ldots, M \} $ with law $ q ( y_1, t_1, \cdot ) , $ where   
\begin{equation}\label{eq:probaq}
 q( y , t,  \{ i \} ) =  \frac{ f_i ( \sum_j y^{(ij)}) }{  \lambda^*_{t} } , ~1 \le i \le M,~ q(  y  ,t, \{ 0 \} ) = 1 - \sum_{ i=1}^M q (  y  ,t,\{i\} ).
\end{equation} 
$U_1$ gives the label of the particle (or component) that will jump at time $t_1,$ if $ U_1 = 0, $ no jump happens at this time. More precisely, on $\{ U_1 \geq 1 \}, $ we put 
$$x_{t_1 } := x_1 + a (x_1) , \; y^{(ij )}_{t_1} := y^{(ij)}_1 +c_{ij} \one_{\{  U_1 = j \}}  , \mbox{ for all } 1 \le i, j \le M , $$
and $ z_{t_1} := z_1 $ on $ \{ U_1 = 0 \} .$ 

Then we put
$$x_s =  \Phi_{s  - t_1} ( x_{t_1} ) ,  y_s = \varphi_{s-t_1} ( y_{t_1} )  \mbox{ for all } t_1 \le s < t_2 , $$
and we proceed iteratively by choosing, conditionally on $y_{t_2 -}  = y_2, $ a random variable 
$$ U_2 \sim q ( y_2, t_2, \cdot ) ,$$
and so on. Finally, we obtain a terminal value $x_t  = \Phi_{t - t_n } ( x_{t_n} ) $ and $ y_t = \varphi_{t - t_n } ( y_{t_n} ).$   It is easy to check that 
$$ {\mathcal L} ( z_t ( z,t_1^n ,  U_1^n  )) = {\mathcal L} ( Z_t |  N_t^* = n ,  T_1^* = t_1, \ldots,  T_n^* = t_n) .$$
The important point is that the above construction ensures the continuity of the application
$$ z \mapsto z_t ( z, t_1^n ,  U_1^n  ).$$

\begin{proof}[Alternative proof of Proposition \ref{prop:Feller}]
Noticing that the law of $ N^* $ does not depend on the starting point $z$ but only on an upper bound of the associated intensities, the above construction implies that for any $ \Phi \in C_b ( \R \times \R^{M \times M}) , $ the mapping  
$$ K \ni z \mapsto P_t \Phi  (z) \in \R $$
is continuous, 
implying the assertion of Proposition \ref{prop:Feller}. Indeed, we can write 
\begin{multline*}
 P_t \Phi ( z) = \sum_{n \geq 0} \P( N_t^* = n )   \int_{[0, t ]^n } \mathcal{L} ( T_1^*, \ldots, T_n^* | N_t^* = n ) (dt_1, \ldots, dt_n) \\
\sum_{i=1}^n \sum_{u_i = 0}^M \P ( U_1 = u_1, \ldots, U_n = u_n | T_1 = t_1, \ldots, T_n = t_n , Z_0 = z )  \E^W  \Phi ( z_t (z, u_1^n ) ) ,
\end{multline*}
where the last expectation $ \E^W$ is only taken with respect to the underlying Brownian motion $W.$ Suppose now that $( z^k)_k \subset K_1 \times K_2 $ is a sequence of initial configurations converging to $ z $ as $k \to \infty.$ Then, by the structure of \eqref{eq:probaq}, 
\begin{multline*}
 \P ( U_1 = u_1, \ldots, U_n = u_n | T_1 = t_1, \ldots, T_n = t_n , Z_0 = z^k ) \to \\
  \P ( U_1 = u_1, \ldots, U_n = u_n | T_1 = t_1, \ldots, T_n = t_n , Z_0 = z ),
\end{multline*}   
as $ k \to \infty ,$  by continuity of  $ f_i, 1 \le i \le M,$ and of $ z \mapsto z_s ( z, t_1^n ,  U_1^n  ) ,$ for all $ s \le t.$    Moreover, using dominated convergence,  $ \E^W  \Phi ( z_t (z^k , u_1^n ) ) \to \E^W  \Phi ( z_t (z, u_1^n ) ),$  as $ k \to \infty, $ implying that $ P_t \Phi (z^k ) \to P_t \Phi (z) ,$ as $k \to \infty, $ where we have used dominated convergence once more.
\end{proof}

\section{Ergodicity of the process $ (X, Y)$}\label{sec:proba}
In this  section, we start proving the positive Harris recurrence of $Z_t= (X_t , Y_t)$ using the regeneration technique.

\subsection{A D\"oblin type lower bound}

The following theorem is the main result of this section. It states a local D\"oblin lower bound for the transition semigroup of the process, which is the main ingredient towards ergodicity of the process $(Z_t)$. 

\begin{theo}\label{thm:Doeblin}
Grant Assumptions \ref{ass:coeff}  and  \ref{ass:markov} and suppose moreover that $ f_i ( x) > 0 $ for all $ 1 \le i \le M, $ for all $ x \in \R.$ Then there exists $ T > 0 $ such that for all $z^* = (x^*, y^*)  \in  \R \times \R^{M \times M }$ and for all $ x^{**} \in \R $ the following holds. There exist $R > 0 , $ open sets $ I_1 \subset \R $ and $I_2 \subset \R^{M \times M} $ with strictly positive Lebesgue measure such that 
\begin{equation}
  x^{**} \in I_1,
\end{equation}    
and a constant $\beta \in (0, 1), $ depending on $I_1, I_2, R$ and the coefficients of the system with  
\begin{equation}\label{doblinminorization}
P_{T} (z , dz' ) \geq \beta \one_C  (z) \nu ( dz') ,
\end{equation} 
where $ C = B_R ( z^* ) $ is the (open) ball of radius $R$ centred at $z^* ,$ and where $ \nu $ is the uniform probability measure on $  I_1 \times I_2 .$ 
\end{theo}   

\begin{rem}
1. Together with a Lyapunov argument implying that the process $Z$ comes back to the set $C$ infinitely often, the above result will imply the $\nu-$irreducibility of the sampled chain $ (Z_{kT})_{ k \geq 0 }.$ We want to stress that the above construction implies  that for any given $ x^{**} \in \R $ and some $\varepsilon > 0 $ we can construct $ \nu $ such that $ B_\varepsilon (x^{**} ) $ lies in the support of the projection on the $x-$variable of $ \nu.$ Of course, this property is related to the support properties of the underlying diffusion process, granted by Assumption \ref{ass:coeff}. It will imply that the invariant density of $X$ is bounded away from $ 0 $ on each compact set, see Proposition \ref{prop:supportpi} below.\\
2. A lower bound of the type \eqref{doblinminorization} is trivial if we are only interested in the transitions of the process $X.$ Indeed, in this case it is sufficient to consider the event where no jump has appeared up to time $T$ and to use known lower bounds on the transition densities of the diffusion part.  However, since $X$ is not Markov on its own, we do need to work with the couple of processes $(X, Y) , $ and therefore have to establish such lower bounds for the joint transition of $ X$ and $Y.$ 
\end{rem}

The main idea of the proof is to use the jump noise of $M^2$ successive jumps of the underlying Hawkes process to create firstly a Lebesgue density for the process $Y. $ Such ideas have already been exploited in \cite{Clinet} and in \cite{DLO16}. The important  second step is then to use density estimates of the underlying diffusion to prove that a joint Lebesgue density of $ (X_T,Y_T) $ exists which is strictly lower-bounded on $ C.$ 

The proof is done in several steps which are the subject of the next subsection. 

\subsection{Some useful properties of the underlying stochastic flow and proof of Theorem \ref{thm:Doeblin}}
We start by collecting useful properties of the stochastic flow governing the evolution of $X$ in between successive jumps. \\
{\it Transition densities.}
Due to our assumptions on $b$ and $ \sigma$, the stochastic flow $ \Phi_t ( x) $ given by Equation \eqref{eq:flow}, possesses a transition density $x \mapsto p_t ( x, y ) $ with explicit lower bounds. More precisely, Proposition 1.2 of \cite{gobet} implies that,  for some suitable constants $c, C, $ 
\begin{equation}\label{eq:aronson}
c^{-1}  t^{- 1/2} e^{- C (x-y)^2 /t}   e^{- C t x^2 }  \le  p_t (x, y ) \le c t^{- 1/2} e^{- \frac{1}{C} (x-y)^2 /t} e^{ C t x^2 } .
\end{equation}
Here, the constants $c$ and $C$ do only depend on the coefficients $ b $ and $\sigma .$ \footnote{It is possible to replace our Condition \ref{ass:coeff} by any other condition ensuring that $p_t (x,y ) $ is strictly lower bounded on compact sets for any $t > 0 $ .} 

{\it Support properties.} We will use the control theorem which goes back to Stroock and Varadhan (1972) \cite{StrVar-72}, see also Millet and Sanz-Sole (1994) \cite{MilSan-94}, Theorem 3.5. For some time horizon $T_1<\infty$ which is arbitrary but fixed, write $\,\tt H\,$ for the Cameron-Martin space of measurable functions ${\tt h}:[0,T_1]\to \R $ having absolutely continuous components ${\tt h} (t) = \int_0^t \dot{\tt h} (s) ds$ with $\int_0^{T_1}(\dot{\tt h})^2(s) ds < \infty .$ For $x\in \R $ and ${\tt h}\in{\tt H}$, consider the deterministic flow
\begin{equation}\label{generalcontrolsystem}
(\varphi_t^{( {\tt h}) } (x) )_{t \geq 0 }  \; \; \mbox{solution to}\; \; d \varphi^{( {\tt h}) }_t (x) = \tilde b (  \varphi^{( {\tt h}) }_t (x) ) dt + \sigma( \varphi^{( {\tt h})_t } (x) ) \dot{\tt h}(t) dt, \; \mbox{with $\varphi_0^{( {\tt h })} (x)=x,$}  
\end{equation}
on $[0, T_1 ].$ In the above formula \eqref{generalcontrolsystem}, $\tilde b$ is Stratonovich drift given by 
$$ \tilde b ( x) = b(x) - \frac12 \sigma (x) \sigma' ( x) .$$ 
Denote by $ Q_x^{T_1} $ the law of the solution $ ( \Phi_t (x) )_{0 \le t \le T_1} .$ Then for any $ x \in \R  $ and  $  {\tt h } \in \tt H ,$ 
\begin{equation}\label{eq:support}
 \left(\varphi_t^{({\tt h}  ) } (x) \right)_{|t \in  [0,  T_1 ] } \in \overline{ {\rm supp} \left( Q_x^{T_1 } \right)}.
\end{equation}

We are now able to give the proof of Theorem \ref{thm:Doeblin}. 

\begin{proof}
We suppose without loss of generality that for all $ 1 \le j \le M, $ 
$$ c_{ij} \neq 0 \mbox{ for all } i , \; \alpha_{ij } \neq \alpha_{kj},  \mbox{  for all $ 1 \le i, k \le M, i \neq k .$} \quad \footnote{Otherwise, if $ \alpha_{ij} = \alpha_{kj}, $ we introduce the new process $ Y_t^{ (i+k),  j} := Y_t^{ij} + Y_t^{kj} $ which is Markov again.}   $$

In what follows, we impose first $M$ consecutive jumps of $ N^{(1)}, $ followed by $M$ jumps of $ N^{(2)} $ etc up to $M$ consecutive jumps of $ N^{(M)}, $ and we suppose that they all happen before time $T.$  
Then we can lower-bound, for all $ A \in {\cal B} ( \R ), B \in  {\cal B} ( \R^{M\times M }),$ $z \in \R \times \R^{M\times M},$ 
\begin{multline*}
 P_T ( z, A\times B ) \geq  \E_z \left( \one_A (X_T) \one_B  ( Y_T), N^{(i)}_T =M, \forall 1 \le i \le M, \right.\\
\left. T_1^{(1)} < T_2^{(1)} < \ldots < T_M^{(1)} < T_1^{(2)} < \ldots < T_M^{(2)} < \ldots < T_1^{(M)} < \ldots  T_M^{(M)} < T \right) ,
\end{multline*} 
where $ T_k^{(i) }, k \geq 1, $ are the successive jump times of $ N^{(i)}, 1 \le i \le M.$  In what follows we shall write 
\begin{multline*}
 A_T :=   \{N^{(i)}_T =M,~ \forall 1 \le i \le M, \\
  T_1^{(1)} < T_2^{(1)} < \ldots < T_M^{(1)} < T_1^{(2)} < \ldots < T_M^{(2)} < \ldots < T_1^{(M)} < \ldots  T_M^{(M)} < T \} .
\end{multline*}   

{\it Step 1.} We first deal with the process $ Y_t .$ This part of the proof follows well-known arguments that have been developed independently from our work in \cite{Clinet}, and that we have also already used in \cite{DLO16} and \cite{evaSPA}. We reproduce the main arguments here since the second step will be a control of the diffusion part conditionally on the results of this first step.

Recall that in between successive jump events of $N, $ the process $Y$ evolves according to the deterministic flow $\varphi_t (y) $ introduced in \eqref{eq:flowygrand} above. Thus, on the event $A_T , $ starting from $Y_0 = y \in \R^{M \times M}  , $ we first let the flow $\varphi $ evolve starting from $ y  $ up to some first jump time $t_1 .$ At that jump time each particle having index $ (i1) $ gains an additional value $c_{i1} .$ We then successively choose the following inter-jump waiting times $t_2, \ldots, t_{M^2}$ under the constraint $ t_1 + \ldots + t_{M^2} < T .$ We write 
$$ s_1= T- t_1, s_2= T - (t_1 + t_2) , \ldots,~ s_{M^2}  = T - (t_1 + \ldots + t_{ M^2}) .$$ 

Conditionally on $Y_0=y$, the successive choices of ${\bf s}=(s_1, \ldots , s_M^2 )$ as above, the position of $ Y^{(ij)}_T $ is given by 
\begin{equation*}
\gamma^{(ij )} (y,  {\bf s})= e^{ - \alpha_{ij} T} y^{(ij)} +c_{ij} [  e^{- \alpha_{ij}   s_{ jM - M +1} } +\ldots +  e^{- \alpha_{ij}  s_{ jM}} ] .
\end{equation*}

In what follows we work at fixed $ y$ and we write  
$$ \gamma_{ y } :    {\bf {s}} \mapsto \gamma (y,  {\bf {s}}) = ( \gamma^{(11)}, \ldots , \gamma^{(M1)}, \gamma^{(12)}, \ldots , \gamma^{(M2)}, \ldots ,\gamma^{(1M)}, \ldots ,  \gamma^{(MM)} ) (y,  {\bf {s}}). $$

We will use the jump noise which is created by the $M^2$ jumps, i.e., we will use a change of variables on the account of $s_1, \ldots, s_{M^2} .$ 
Therefore, let 
$$
\frac{\partial \gamma_{y}({\bf {s}})}{\partial {\bf {s}}}=\Big[\frac{\partial \gamma_{y }({\bf {s}})}{\partial s_1},\ldots, \frac{\partial \gamma_{y}({\bf {s}})}{\partial s_{M^2}} \Big]
$$ 
be the Jacobian matrix of the the map ${\bf s}\mapsto \gamma_{y}({\bf {s}})$.  
This matrix does not depend on the initial position $y.$  Indeed, one easily finds that
$$
\frac{\partial \gamma_{y}({\bf {s}})}{\partial {\bf {s}}}=C ( {\bf {s}} ) =   \left( 
\begin{array}{ccccc}
C^{(1)} & {\bf 0} & \cdots &  {\bf 0} &  {\bf 0} \\
{\bf 0} & C^{(2)} &   {\bf 0}  & \cdots &   {\bf 0} \\
\vdots & \vdots & \vdots & \ddots & \cdots \\
  {\bf 0} & \cdots &   {\bf 0}  &  {\bf 0}    & C^{(M)} 
  \end{array} 
\right)  ( {\bf {s}} ),
$$ 
where for each $1\leq j\leq M$, $C^{(j)}( {\bf {s}} ) $ is the  $M \times M$ matrix given by 
$$
C^{(j)}( {\bf {s}} ) =- 
\left( \begin{array}{ccc }
 c_{1j} (\alpha_{1j})^{-1}  e^{ - \alpha_{1j} s_{jM - M+1} } & \ldots &  c_{1j} (\alpha_{1j})^{-1}  e^{ - \alpha_{1j} s_{jM } } \\
c_{2j} (\alpha_{2j})^{-1}  e^{ - \alpha_{2j} s_{jM - M+1} } & \ldots &  c_{2j} (\alpha_{2j})^{-1}  e^{ - \alpha_{2j} s_{jM } } \\
\vdots & \vdots & \vdots  \\
c_{Mj} (\alpha_{Mj})^{-1}  e^{ - \alpha_{Mj} s_{jM - M+1} } & \ldots &  c_{Mj} (\alpha_{Mj})^{-1}  e^{ - \alpha_{2j} s_{jM } } \end{array}
\right).
$$
We have to prove that each $C^{(j)}  ( {\bf {s}}) $ is invertible. This is difficult for general choices of ${\bf s } .$ But if we choose $ t_0$ such that $M^2 t_0 < T $ and then $ s^*_1 = M^2 t_0 , s^*_2 = (M^2 - 1) t_0, \ldots, s^*_{M^2 } = t_0 $ such that $ s_k - s_{k+1} = t_0 $ for all $ 1 \le k < M^2,$ then 
$$ \det C^{(j)} ( {\bf s^* } ) = \left( \prod_{i=1}^M c_{ij} (\alpha_{ij})^{-1} e^{- \alpha_{ij} s^*_{jM } }\right)  \det V^{(j)} ( t_0 ) ,$$
where 
$$ V^{(j)} (t_0) = \left( 
\begin{array}{ccccc}
e^{- (M-1) \alpha_{1j} t_0 }   &  \ldots  & e^{- 2 \alpha_{1j} t_0 } &  e^{- \alpha_{1j} t_0} & 1 \\
e^{- (M-1) \alpha_{2j} t_0 }   &  \ldots  & e^{- 2 \alpha_{2j} t_0 } &  e^{- \alpha_{2j} t_0} & 1 \\
\vdots & \vdots & \vdots & \vdots & \vdots \\
e^{- (M-1) \alpha_{Mj} t_0 }   &  \ldots  & e^{- 2 \alpha_{Mj} t_0 } &  e^{- \alpha_{Mj} t_0} & 1
\end{array}
\right) .$$
Therefore, $ \det V^{(j)} (t_0) $ is a Vandermonde determinant which is different from $0$ since by assumption all $ e^{-  \alpha_{ij} t_0} $ and $ e^{ - \alpha_{kj} t_0 } $ are different, for all $ i \neq k.$ 

By continuity we therefore have that for any choice of $ t_0 < T/M^2 $ there exists $ \varepsilon > 0 $ such that for all $ {\bf s} \in B_\varepsilon ( {\bf s^*} ) , $ $\frac{\partial \gamma_{y}({\bf {s}})}{\partial {\bf {s}}} $ is invertible.

It will be proved now that this uniform invertibility of the Jacobian matrix of the map ${\bf s}\mapsto \gamma_{y }({\bf {s}})$ implies the first part of inequality \eqref{doblinminorization}. For that sake, we shall also need the following notation. For each couple $(y, {\bf {s}})$, we write $y_0 = y , $ and then define recursively for all $ jM - M + 1  \le k \le jM ,$ where $j$ varies between $ 1 \le j \le  M, $ 
\begin{equation}\label{eq:sautsy}
y^{(i \ell)}_k=\varphi^{(i\ell)} _{s_{k-1}-s_{k}}(y_{k-1})+   c_{i j} \one_{ \{ \ell = j \}} ,
\end{equation}
for all $ 1 \le i , \ell \le M, $ where we put $s_0 = T.$ 

The sequence $y_1,\ldots ,y_{M^2}$ corresponds to the positions of the process $Y$ right after the successive jumps, starting from the initial location $y\in \R^{M \times M},$ induced by  the inter-jump waiting times $T-s_1,s_1-s_2, \ldots , s_{M^2-1} -  s_{M^2} $ which are determined by ${\bf {s}}$. 

Introduce now for each $y\in \R^{M \times M}$
the total jump rate 
$$\bar f( y) := \sum_{ i=1}^M f_i \left( \sum_j  y^{(ij)}\right) $$
and for each $t\geq 0,$  the survival rate 
\begin{equation}
\label{def:suvival_rate}
 e(y,t)= \exp\Big\{-\int_{0}^t \bar f\big(\varphi_s^{}(y)\big)ds\Big\} .
\end{equation}
We  define  for each couple  
$(y,  {\bf s})$ (recall that $s_0=T$), 
\begin{equation}\label{densityat_x_c_s}
q_{y}({\bf {s}})=\left( \prod_{j=1}^{M} \prod_{k= jM - M}^{jM - 1 } \; f_j (   \varphi_{s_k-s_{k+1}}(y_k) )  \; e(y_k,s_k-s_{k+1}) \right) e(y_{M^2},s_{M^2}).
\end{equation}

Since $\bar f (y) > 0 $ for all $ y \in \R^{ M \times M}$ and from the definition of $e(\cdot, \cdot),$ we deduce that for any couple $(y^*,{\bf {s}^*})$ there are neighborhoods $W_{\bf {s}^*}$ and $U_{y^*}$ of ${\bf s}^* $ and ${y^*}$ respectively such that  
\begin{equation}
\label{lowerboundq}
\inf_{(y, {\bf s})\in U_{y^*}\times  W_{\bf {s}^*}} q_{ y }({\bf {s}})>0. 
\end{equation}   
Let us now fix $(y^*,{\bf {s}}^{*})$ such that the matrix $\frac{\partial \gamma_{y^*}({\bf {s}}^*)}{\partial {\bf {s}}}$ is invertible. By Lemma 6.2 of \cite{benaim2015}, there exist an open neighborhood $ {B}_{R}(y^*)$ of  $y^* , $ an open set $I_2 \subset \R^{M \times M }$, and for any $y\in  {B}_{R}(y^*) ,$ an open set $W_{y}$ such that
$$
\tilde{\gamma}_{y}({\bf {s}}):\left\lbrace
\begin{array}{l}
W_{y} \to I_2 \\
{\bf {s}} \mapsto \gamma_{y}({\bf {s}})
\end{array}
\right.
$$
is a diffeomorphism, with $W_{y}\subset W_{\bf {s}^*},$ and also
\begin{equation}
\label{lowerbounddet}
\inf_{y \in  {B}_{R}(y^*) } \inf_{\bf {s}\in W_{y} }\Big|\mbox{det}\Big(\frac{\partial \gamma_{y}({\bf {s}})}{\partial {\bf {s}}}\Big)^{-1}\Big|>0.
\end{equation}
Reducing (if necessary)   
$R$, we may assume also that $ {B}_{R}(y^*) \subset U_{y^*}$. Thus we have that by  \eqref{lowerboundq} and \eqref{lowerbounddet}, 
\begin{equation}
\label{lowerbounddet2}
 \beta_1 := \inf_{y\in  {B}_{R}(y^*) } \inf_{{s} \in W_{y}} q_{y}({\bf {s}}) \Big|\mbox{det}\Big(\frac{\partial \gamma_{y}({\bf {s}})}{\partial {\bf {s}}}\Big)^{-1}\Big|>0.
\end{equation}
Once we have done all these steps we can conclude with the following preliminary result. For any measurable
$B \in {\mathcal B}(\R^{M\times M})$ and for any $z = (x, y ) $ such that $y\in  {B}_R (y^*), $ using the change of variables $y=\tilde{\gamma}_{y}({\bf {s}})$, we obtain the lower bound
\begin{multline}\label{eq:tobe}
\E_y \left( \one_B ( Y_T) , N^{(i)}_T =M, \forall 1 \le i \le M, T_1^{(1)} < \ldots < T_M^{(1)} < T_1^{(2)} < \ldots < T_M^{(M)} < T\right) \\
\geq  \int_{W_{y}}  q_{y}({\bf {s}})\one_{B}(\tilde \gamma_{y}({\bf {s}}))ds_{1}\ldots ds_{M^2} 
\geq  \beta_1  \int_{I_2 \cap B}dy_1\ldots dy_{M^2}.
\end{multline}
This would establish the desired result if we were only interested in the autonomous process $Y.$ 

{\it Step 2.} We now deal with the process $X.$ Of course, we still work conditionally on the choice of $ s_1, \ldots, s_{M^2}$ of the first step. Analogously to \eqref{eq:sautsy} we therefore introduce the successive jump positions of the process $X$ which are given by $x_0 = x $ and then for all $ 1 \le k \le M^2, $
\begin{equation}\label{eq:sautsx}
x_k=\Phi_{s_{k-1}-s_{k}}(x_{k-1})+ a ( \Phi_{s_{k-1}-s_{k}}(x_{k-1})),
\end{equation}
where $s_0  = T $ as before. 

Conditionally on $X_0 = x , Y_0=y$, the successive choices of ${\bf {s}}=(s_1, \ldots , s_{M^2} )$ as above, the position of $ X_T $ is then given by 
\begin{equation*}
\Gamma  (z,  {\bf {s}})=  \Phi_{ s_{M^2}} ( x_{M^2}) ,
\end{equation*}
where $ z = (x, y ) .$ Notice that $ x_{M^2} $ depends on the choices of $ s_1, \ldots, s_{M^2} $ and of course on the evolution of the stochastic flow between the successive jump times.  

{\it Step 3.} Therefore, conditioning with respect to $ X_{T - s_{M^2}}  =x_{M^2} , $ we obtain 
\begin{multline}\label{eq:xm}
P_T ( z, A\times B ) \geq 
\E_z \left( \one_A (X_T) \one_B  ( Y_T), A_T  \right)x \\
 \geq  \int_{W_{y}}  q_{y}({\bf {s}}) \one_{B}(\tilde \gamma_{y}({\bf {s}})) \; \E \left( \int_A p_{s_{M^2}} ( x_{M^2}, u ) du \right) ds_{1}\ldots ds_{M^2},
\end{multline}
where $p_{s_{M^2}} ( x_{M^2}, u )$ is the transition density of \eqref{eq:aronson}. 

Notice that in the last line, expectation $ \E $ is taken with respect to Brownian motion only, that is, with respect to the law of $X_{T- s_{M^2}} $ under the Wiener measure, conditionally on the choices $ s_1, \ldots, s_{M^2}.$ 

The lower bound estimate given in \eqref{eq:aronson} implies that for fixed $s_{M^2}$ and $x_{M^2}, $  $ u \mapsto  p_{s_{M^2}} ( x_{M^2} , u )$ is lower bounded in small neighborhoods of $x_{M^2}.$ Therefore, in what follows we need to localize the possible positions of $x_{M^2}, $ conditionally on the choices of the jumps times $s_1, \ldots , s_{M^2}.$ 

{\it Step 4. Localization of $x_{M^2}.$} 

We apply the support theorem, that is, \eqref{eq:support}, to our process $X_t$ in between the successive jump times $ 0, T_1 := t_1,T_2 :=  t_1+ t_2, \ldots, T_{M^2} := t_1 + \ldots t_{M^2} , $ by choosing on each time interval $ [ T_n, T_{n+1}[ , n = 0, \ldots, M^2 - 1, $ the control function $ \tt h \equiv 0 .$ Consequently, introducing 
$$ \tilde x_1 := \varphi^{({\tt 0}  )}_{t_1} (x)  + a (\varphi^{({\tt 0}  )}_{t_1} (x)) , \ldots, \tilde x_{M^2} := \varphi^{({\tt 0}  )}_{s_{M^2-1} - s_{M^2} } (\tilde x_{M^2 - 1} )  + a (\varphi^{({\tt 0}  )}_{s_{M^2-1} - s_{M^2} } (\tilde x_{M^2- 1} )),$$
there exists an open neighborhood $ U_{\tilde x_M } $ of $ \tilde x_{M^2}, $ such that 
$$ \P ( X_{ T- s_{M^2}} \in  U_{\tilde x_{M^2} } ) > 0.$$ 
Notice that $ \tilde x_{M^2} $ is a continuous function of  the starting point $x$ and of $ {\bf {s} } ; $ that is, $   \tilde x_{M^2}  = F (x, {\bf  s} )  $ for some continuous function $F.$ This implies that for any starting point $x^* $ and for $R > 0 $ sufficiently small, reducing $W_{\bf {s}^*}$ if necessary, there exists a compact $ K = K ( x^*, {\bf  s^*} ) $ such that $ F (x, {\bf  s} ) \in K $ for all $ x \in B_R ( x^* ) $ and for all $ {\bf  s} \in W_{\bf {s}^*} ,$ whence 
$$ \inf_{ x \in B_R ( x^* ) } \inf_{  {\bf  s} \in W_{\bf {s}^*}} \P ( X_{T - s_{M^2}} \in J_1 ) > 0, $$
where 
$$ J_1 = \bigcup_{ \tilde x_{M^2} \in K} U_{\tilde x_{M^2}} .$$
Notice that $ J_1$ has compact closure.  

Let now $ x^{**} \in \R $ be arbitrarily chosen. The lower bound of \eqref{eq:aronson} implies that there exists an open interval $I_1 \ni x^{**},$  such that 
$$ \inf_{ x \in J_1 , y \in I_1, {\bf  s} \in W_{\bf {s}^*} } p_{s_{M^2}} ( x, y )  > 0.$$
Therefore,
\begin{equation}
\beta_2 :=  \inf_{ x \in B_R ( x^* ) }  \inf_{y \in  {B}_{R}(y^*) } \inf_{{s}\in W_{y} } \inf_{ u \in J_1, v \in I_1 }  p_{ s_{M^2} } ( u, v) \P_x  ( X_{T - s_M} \in J_1  )  > 0 .
\end{equation}

Therefore, coming back to \eqref{eq:xm}, for all $ z \in B_R ( x^*, y ^* ) , $ 
\begin{eqnarray*}
P_T ( z, A\times B ) &\geq & \E_z \left( \one_A (X_T) \one_B  ( Y_T), N^{(i)}_T =M, \forall 1 \le i \le M, T_1^{(1)} < T_2^{(1)} < \ldots < T_M^{(M)} < T \right) \\
& \geq &  \int_{W_{y}}  q_{y}({\bf {s}}) \one_{B}(\tilde \gamma_{y}({\bf {s}})) \; \E \left( \one_{\{  X_{T- s_{M^2}} \in J_1 \} }\int_A p_{s_{M^2}} ( X_{T- s_{M^2}} , u ) du \right) ds_{1}\ldots ds_{M^2}\\
& \geq &\int_{W_{y}}  q_{y}({\bf {s}})\one_{B}(\tilde \gamma_{y}({\bf {s}})) \; \E \left( \one_{\{  X_{T- s_{M^2}} \in J_1 \} } \int_{ A \cap I_1 } \inf_{ x \in J_1 } p_{s_{M^2}} ( x, u ) du \right) ds_{1}\ldots ds_{M^2} \\
&\geq & \beta_2 \lambda ( A \cap I_1 )  \int_{W_{y}}  q_{y}({\bf {s}})\one_{B}(\tilde \gamma_{y}({\bf {s}}))  ds_{1}\ldots ds_{M^2} \\
&\geq & \beta_2 \lambda ( A \cap I_1 )  \beta_1 \lambda^{M \times M} ( B \cap I_2 ) ,  
\end{eqnarray*}
where we have finally applied \eqref{eq:tobe}, and where $ \lambda $ and $\lambda^{M \times M}$ denote the Lebesgue measure on $ \R, $ $\R^{M \times M}, $ respectively. This implies the desired result putting 
$$ \beta := \beta_1 \beta_2 \lambda (I_1) \lambda^{M \times M}  ( I_2 ) \; \; \mbox{ and } \; \nu := \mathcal{U}_{ I_1  \times I_2  },$$
the uniform probability law on $  I_1 \times I_2 .$ 
\end{proof} 

\subsection{A Foster-Lyapunov type condition }

In order to prove the positive Harris recurrence of the process $ Z, $ we need of course a stability condition which is a Lyapunov type condition. 
Notice that the process $Z_t = ( X_t, Y_t)$ has the following extended generator, defined for sufficiently smooth test functions $g$ by 
\begin{eqnarray*}
\mathcal{A}^{Z}g(x,y)&=& - \sum_{i, j =1}^M \alpha_{ij}  y^{(ij)}\partial_{y^{(ij )}} g(x,y) + \partial_{x} g(x,y) b(x) + \frac{1}{2}  \sigma^2(x)\partial^2_{x} g(x,y)\\
&&
+\sum_{j=1}^{M} f_j \left( \sum_{k=1}^M y^{(jk)} \right) \left[g \left(x+a(x), y+\Delta_{j } \right)-g(x,y) \right] , 
\end{eqnarray*}
with $(\Delta_{j})^{(il)} =  c_{ij} \one_{\{ j=l\}} $ for all $ 1 \le i, l \le M.$ 

To obtain stability we first introduce the following classical stability assumption for Hawkes processes. Recall that $ \gamma_i$ denotes the Lipschitz constant of the rate function $f_i.$ 
\begin{ass}\label{ass:h}
Let $H $ be the $M \times M-$matrix with entries $H_{ij} = \gamma_j  \frac{|c_{ij} | }{ \alpha_{ij} }, 1 \le i, j \le M. $ Then its spectral radius satisfies 
$$ \rho := \rho ( H ) <1. $$ 
\end{ass}

Under the above stability condition, let $ \kappa \in \R_+^M $ be a left eigenvector of $ H,$ associated to the eigenvalue $ \rho  $ and having non-negative components,  that is, for all $j, $ $ \sum_i \kappa_i H_{ij} = \rho \kappa_j  \geq 0 .$ 
Such a vector $ \kappa $ exists thanks to the theorem of Perron-Frobenius. 
Following \cite{Clinet}, we introduce a Lyapunov function by defining  $ m_{ij}= \frac{\kappa_i}{\alpha_{ij} } $ for all $ 1 \le i, j \le M$ and $ V : \R \times \R^{M \times M} \to \R_+ $ by 
\begin{equation}
V(x, y ) =V_1 (x)  + e^{ \sum_{i, j } m_{ij} | y^{(ij)} | } ,
\end{equation}
where $ V_1 : \R \to \R_+ $ will be chosen in the sequel. Notice that $ V (x, y ) \geq 1 $ for all $x, y .$

\begin{prop}\label{prop:lyapunov}
Grant Assumptions \ref{ass:coeff}, \ref{ass:erg-expo} (Exponential frame) and \ref{ass:h}. Let $ V_1 ( x) = x^2.$ Then there exist positive constants $d_1, d_2$ such that the following Foster-Lyapunov type drift condition holds
\begin{equation}\label{eq:lyapunov}
\mathcal{A}^Z V\leq  d_1 - d_2 V  .
\end{equation}
\end{prop}

\begin{proof}
We write $ V(x, y ) = V_1 (x) + V_2 (y ) ,$ where $ V_1 (x) =  x^2, $ $V_2 (y) = e^{ \sum_{i, j } m_{ij} |y^{(ij)} | } ,$ and $ \mathcal{A}^Z V = \mathcal{A}_1^Z V + \mathcal{A}_2^Z V , $ where 
$$ \mathcal{A}_1^Z V  (x, y ) =  2 x  b(x) +   \sigma^2(x) +\sum_{j=1}^{M} f_j \left( \sum_{k=1}^M y^{(jk)} \right) [ (x + a(x) )^2 - x^2 ] $$ 
is the diffusion part and 
$ \mathcal{A}_2^Z V  =  \mathcal{A}^Z V -  \mathcal{A}_1^Z V $ is the jump part of the generator. The arguments of the proof of Proposition 4.5 of \cite{Clinet} imply that 
\begin{equation}\label{eq:good}
 \mathcal{A}_2^Z V (x, y )  = \mathcal{A}_2^Z V_2 (y)  \le - c_1 V_2 (y) +  c_2 \one_{K_1} (y) , 
\end{equation} 
where $ c_1, c_2  > 0 $ and where $ K_1  \subset \R^{M \times M} $ is some compact. Moreover, writing for short $ \bar f(y) = \sum_i f_i ( \sum_j  y^{(ij)} ) $ for the total jump rate,
\begin{equation}\label{eq:jumpdiff}
 \mathcal{A}_1^Z V (x,y)  =  2 x b(x) + \sigma^2 (x) + ( 2 x a(x) + a(x)^2 ) \bar f (y) .
\end{equation} 
Thanks to Assumption \ref{ass:erg-poly} and bound of $\sigma^2$ from Assumption \ref{ass:coeff}, 
\begin{equation}\label{eq:goodgood}
 2 x b(x) + \sigma^2 (x) \le -  c_3 x^2 + c_4 \one_{K_2} ( x) ,
\end{equation}
where $ K_2 \subset \R $ is a compact set. Now, suppose that Assumption \ref{ass:erg-expo} is satisfied. Then for all
$|x| > r, $ $ ( 2 x a(x) + a(x)^2 ) \bar f(y)   \le 0  $ (recall that all $\bar f  $ is non-negative) implying that 
$$  \mathcal{A}_1^Z V (x,y) \le - c_3 x^2 + c_4 \one_{K_2}  ( x) + c_5 \bar f(y)  .$$ 
This implies \eqref{eq:lyapunov} since 
$$\bar f(y)  \le c_6 + c_7 \sum_{i,j } |y^{(ij)} |  \le c_6 + \tilde c_7 \log ( V_2 (y) ) $$
which is thus negligible with respect to the negative term  $ - c_1 V_2 (y) $ of \eqref{eq:good}.  

If only Assumption \ref{ass:erg-poly} holds, then we upper bound the jump part of \eqref{eq:jumpdiff} by 
$$( 2 x a(x) + a(x)^2)  \bar f (y) \le  C \bar f( y) + C  |x|^{1 + \eta } \bar f (y ) .$$
Choose $1 <  p < 2   $ and $ q > 2$ such that $ (1 + \eta ) p < 2 $ and $ \frac{1}{p} + \frac{1}{q} = 1.$ Then 
$$   |x|^{1 + \eta } \bar f (y )  \le  \frac{1}{p} |x|^{p ( 1 + \eta)}  + \frac{1}{q} \bar f^q (y)  .$$
Since $ (1 + \eta ) p < 2,$ the first term 
$$C  \frac{1}{p} |x|^{ (1 +\eta ) p} $$ 
will again be negligible with respect to the negative term $ - c_3 x^2 $ of \eqref{eq:goodgood}. Finally,  the second term $C \bar f( y) + C  \frac{1}{q} \bar f^q (y)   $ is treated as previously. This concludes our proof. 
\end{proof}

\begin{prop}\label{prop:lyapunovbis}
Grant Assumptions \ref{ass:coeff}, \ref{ass:erg-poly} (Polynomial frame) and \ref{ass:h}. Let $ V_1 ( x) = 1+ |x|^m$ and $ \alpha = 2/ m \in ] 0, 1 [.$  Then there exist positive constants $d_1, d_2$ such that the following Foster-Lyapunov type drift condition holds
\begin{equation}\label{eq:lyapunovbis}
\mathcal{A}^Z V\leq  d_1 - d_2 V^{ 1- \alpha}  .
\end{equation}
\end{prop}

\begin{proof}
Using the same arguments as in the proof of Proposition \ref{prop:lyapunov}, we have that 
$$ \mathcal{A}_2^Z V (x, y )  = \mathcal{A}_2^Z V_2 (y)  \le - c_1 V_2 (y) +  c_2 \one_{K_1} (y) .$$
Moreover, for all $ |x| \geq r, $ since $  |x + a(x) |^m -| x|^m \le 0,$ 
\begin{eqnarray*}
 \mathcal{A}_1^Z V  (x, y ) &=&   b(x)  V_1' (x) +   \frac{1}{2} \sigma^2(x)V_1^{''} ( x)  +\sum_{j=1}^{M} f_j \left( \sum_{k=1}^M y^{(jk)} \right) [ |x + a(x) |^m -| x|^m ] \\
 &\leq &  b(x)  V_1' (x) +   \frac{1}{2} \sigma^2(x)V_1^{''} ( x).
 \end{eqnarray*}
A standard calculus, see e.g. \cite{veret} or \cite{dashaevabis}, shows that, for suitable constants $e_1, e_2 > 0, $  
$$   b(x)  V_1' (x) +   \frac{1}{2} \sigma^2(x)V_1^{''} ( x)  \le e_1 - e_2 V_1 (x)^{ 1 - \alpha }.$$ 
Using that $ V_1 (x) \geq 1 $ such that $ (V_1 (x) + V_2 (y ) )^{ 1 - \alpha } \le V_1 (x)^{1- \alpha} + (1 - \alpha ) V_2 (y ) \le V_1 (x)^{1- \alpha } + V_2 (y )  ,$ we deduce from this the drift condition \eqref{eq:lyapunovbis} as in the proof of Proposition \ref{prop:lyapunov}.
\end{proof}

\subsection{Harris recurrence of $Z$}
We do now possess all ingredients to obtain our main results. 

\begin{theo}\label{harrisrec}
Grant Assumptions \ref{ass:coeff}, \ref{ass:erg-expo}, \ref{ass:markov} and  \ref{ass:h} and suppose that for all $ 1 \le i \le M, $ $f_i ( x) > 0 $ for all $ x \in \R.$ Then $ (Z_t)_{t \geq 0} $ is positive Harris recurrent with unique invariant measure $ \pi . $ In particular, for any starting point $z $ and any positive measurable function $g : \R \times \R^{M\times M} \to \R_+, $  as $T \rightarrow \infty, $ $\P_z-$almost surely,
$$ \frac{1}{T} \int_0^T g(Z_s)ds \to  \pi(g).$$
\end{theo}

\begin{proof}[Proof of Theorem \ref{harrisrec}]
1) We fix any $ x^* \in \R $ and we wish to apply Theorem \ref{thm:Doeblin} with $y^* = 0 $ and $ x^* = x^{**}.$ Let $R$ be the associated radius. 

By Proposition \ref{prop:lyapunov} or Proposition \ref{prop:lyapunovbis}, we know that for a suitable compact set $K = K_1 \times K_2 , $ with $ K_1 \subset \R, K_2 \subset \R^{M\times M}, $  $Z$ comes back to $K $ infinitely often almost surely. Moreover, 
$$ \sup_{y \in K_2 , t \geq 0} \| \varphi_t (y) \|_1 := F  < \infty \; \; \mbox{ and } \; \;   \sup_{y \in K_2} \| \varphi_t (y) \|_1  \to 0 $$ 
as $t \to \infty ,$ by the explicit form of the flow in \eqref{eq:flowygrand}. Therefore there exists $t^* $ such that $\varphi_t (y) \in B_{R } ( 0) $ for all $t \geq t^* , $ for all $ y \in K_2 .$

Applying once more the support theorem for diffusions and observing that $ \sigma $ is strictly positive, Equation \eqref{eq:support} implies that
$$ \inf_{ x \in K_1 } \P ( \Phi_{t^*+s } ( x) \in B_R ( x^* ), 0 \le s \le 2T  ) > 0 $$ 
and thus
$$   \inf_{z\in K} \P_z ( X_{t^* + s  } \in B_R (x^* ) , Y_{t^*+s  } \in B_R ( 0 ), 0 \le s \le 2T  )    > 0 . $$ 
Consequently, using a conditional version of the Borel-Cantelli lemma, the sampled Markov chain $(Z_{kT})_{k \in \N}  $ visits $ B_{R } ( x^*, 0 )$ infinitely often almost surely. \\
2) The standard regeneration technique (see e.g. \cite{dashaeva}) allows to conclude that $(Z_{kT})_{k \in \N} $ and therefore $(Z_t)_t $ are Harris recurrent. This concludes the proof. 
\end{proof}

The following by-product of the above result will prove to be useful when dealing with statistical inference within this new model class.

\begin{prop}\label{prop:supportpi}
Grant all assumptions of Theorem \ref{harrisrec}  and write $ \pi^X $ for the projection of the invariant measure $ \pi $ onto the $X-$coordinate, that is, $\pi^X (dx ) = \int_{\R^{M \times M}} \pi (dx, dy ).$ Then $ \pi^X$  possesses a Lebesgue density which is bounded away from zero on each compact of $ \R.$ 
\end{prop}

\begin{proof}
Let $ A \in {\mathcal B} ( \R ) .$ Then for any $t > 0, $ 
\begin{equation}\label{eq:418}
 \pi^X ( A) = \int_{\R \times \R^{M\times M} } \pi ( dz ) \E_z [ \one_A ( X_t) ] . 
\end{equation} 
Let $ L_t := \sup \{ s \le t :  \exists j : \Delta N_s^{(j) } = 1\} $  be the last jump time of the process before time $t,$ $L_t = 0 $ if there is no such jump. Then by Fubini,
$$ \E_z [ \one_A ( X_t) ] = \int_A  \E_z [ p_{ t - L_t} ( X_{L_t } , y ) ] dy ,$$
where $p_t (x, y ) $ is the transition density of \eqref{eq:aronson}.  This implies the existence of the density $ \pi^X ( x) $ which is given by 
$$  \pi^X ( x) =  \int_{\R \times \R^{M \times M} } \pi ( dz ) \E_z [ p_{ t - L_t} ( X_{L_t } , x ) ]$$
for any $ t > 0.$  

Notice that we do not dispose of any regularity results of $ \pi^X ( x) $ with respect to $x.$ Indeed, the upper bound in \eqref{eq:aronson} does not allow to conclude that the almost sure continuity in $x$ of $ p_{ t - L_t} ( X_{L_t } , x ) $ survives the integration $ \pi (dz) \E_z ( \ldots) .$ 

We are however able to prove that $ \pi^X$ is lower bounded on compacts $ K \subset \R .$  For that sake, fix any $x^{**} \in K$ and apply \eqref{doblinminorization} to $ (x^* , y^* ) \in supp ( \pi ) $  such that $ \pi ( B_R ( z^* ) ) = \pi ( C) > 0 $ and to $ x^{**}.$ Then for any measurable $ A \subset I_1 = I_1 ( x^{**} ), $ applying the lower bound of \eqref{doblinminorization}  to \eqref{eq:418} and taking $t= T,$
$$ \int_A \pi^X (x) dx  \geq \beta \pi ( C) \frac{1}{\lambda (I_1)}  \int_A dx $$
implying that 
$$ \inf_{x\in I_1 } \pi^X(x)  \geq \beta \pi ( C)  \frac{1}{\lambda (I_1)} > 0 .$$
Therefore we have just shown that for all $ x^{**} \in K, $ there exists an open interval $ I_1 = I_1 ( x^{**} )$ containing $ x^{**}  \in I_1$ such that $\pi^X $ is strictly lower bounded on $I_1.$  Since we can cover the compact $K$ by a finite collection of such open intervals $ I_1 (x^{**}), $ this implies the desired lower bound of $ \pi^X $ on compacts.  

\end{proof}

In the sequel, following \cite{MT1993}, in any of the two frames (Assumption \ref{ass:erg-expo} or \ref{ass:erg-poly}) and for the choice of $V$ as in Proposition \ref{prop:lyapunov} or Proposition \ref{prop:lyapunovbis}, we introduce
$$ \| \mu \|_{ V} := \sup_{ g : |g| \le  V } | \mu ( g) | , \;  \| \mu\|_{TV} := \sup_{ g : |g | \le 1 } \mu ( g ) .$$ 
It is now straightforward to obtain our second main result.

\begin{theo}[Ergodicity]\label{theo:second}
Grant all assumptions of Theorem \ref{harrisrec}. Then there exist $c_1, c_2 > 0 $ such that for all $z \in \R \times \R^{M \times M},$  under Assumption \ref{ass:erg-expo},
\begin{equation}\label{eq:last} 
\| P_t(z , \cdot )  - \pi\|_{ V} \le c_1  V (z) e^{ - c_2 t} ,
\end{equation} 
and under Assumption \ref{ass:erg-poly}, 
\begin{equation}\label{eq:lastbis} 
\| P_t(z , \cdot )  - \pi\|_{ TV} \le c_1  V (z) t^{ \frac{1}{\alpha} - 1 }, 
\end{equation} 
where $\alpha $ is as in Proposition \ref{prop:lyapunovbis}. 
\end{theo}

\begin{proof}
The sampled chain $ (Z_{kT })_{k \geq 0 }$ is Feller according to Proposition \ref{prop:Feller}. Moreover it is $ \nu-$irreducible, where $ \nu $ is the measure introduced in Theorem \ref{thm:Doeblin}, associated with the point $ ( x^*, 0 ) $ and $x^{**}, $ for any choice of $x^*, x^{**} \in \R,$ used in the proof of Theorem \ref{harrisrec}. Since $\nu$ is the uniform measure on some open set of strictly positive Lebesgue measure, the support of $\nu $ has non-empty interior. Theorem 3.4 of \cite{MT1992} implies that all compact sets are `petite' sets of the sampled chain. Under  Condition \ref{ass:erg-expo}, the Lyapunov condition established in Proposition \ref{prop:lyapunov} allows to apply Theorem 6.1 of \cite{MT1993} which implies the assertion in the exponential frame. In the polynomial frame, the assertion follows from Theorem 3.2 of \cite{dfg} or from Theorems 1.19 and 1.23 in \cite{kulik}. 
\end{proof}

\begin{rem}
In the above Theorem \ref{theo:second}, we do only treat two cases for the rate of convergence to equilibrium: the exponential and the polynomial one. Using slightly different Lyapunov functions, we could also deal with more general sub-exponential rates of convergence, as they have for example been considered in Theorem 1.20 in \cite{kulik}. We do not detail these calculations here since we are mostly interested in exponential rates of convergence. Finally, notice that it is also possible to interpolate between the total variation norm considered in \eqref{eq:lastbis} and the weighted total variation norm $ \| \cdot \|_V $ considered in \eqref{eq:last}, however at the cost of slower rates of convergence. Details can be found in the very instructive paper \cite{dfg}. 
\end{rem}

\subsection{Exponential $\beta$-mixing for $Z= (X, Y).$}
It is now easy to deduce from the above results the exponentially $\beta$-mixing property of the process, under Condition \ref{ass:erg-expo}. 
Recall that the $\beta-$mixing coefficient of $Z$ is given by 
$$ \beta_Z  ( t) = \sup_{s \geq 0} \int \| P_t ( z, \cdot ) - \mu P_{s+t} ( \cdot ) \|_{TV} \; \mu P_s ( dz) ,$$
where $ \mu = {\mathcal L} ( Z_0 )$ is the law of the initial configuration and where 
$$ \| \mu\|_{TV} := \sup_{ g : |g | \le 1 } \mu ( g ) $$
denotes the total variation distance. Notice that if $ \mu = \pi, $ then the process is in its stationary regime, and 
$$ \beta_Z (t) = \int \| P_t ( z, \cdot ) - \pi  \|_{TV} \pi ( dz ).$$

\begin{theo}\label{thm:mixing}
Grant all assumptions of Theorem \ref{harrisrec} and suppose that Condition \ref{ass:erg-expo} holds. Then $Z$ is exponentially $\beta-$mixing, that is, there exist constants $K, \theta  > 0 $ such that for any initial law $ \mu  $ with $ \mu ( V) < \infty, $ 
$$ \beta_Z (t) \le K e^{- \theta t } .$$
\end{theo}

\begin{proof}
Suppose firstly that $ \mu = \pi.$ Then Theorem 4.3 of \cite{MT1993} implies that $ \int V d \pi < \infty $ such that we are able to integrate \eqref{eq:last} against $ \pi ( dx) $ to obtain 
$$ \beta_Z ( t) \le c_1 \pi ( \bar V) e^{-c_2 t }.$$
Putting $ K := c_1 $ and $ \theta = c_2, $ this implies the result in this case. 

In order to deal with the general non-stationary process, we apply Lemma 3.9 of \cite{MASUDA} with $ h = \bar V , $ $\delta (t) = c_1 e^{-c_2 t } $ and 
$$ \kappa = \sup_{ s \geq 0 } \E ( \bar V ( Z_s) ),$$
to deduce that 
$$ \beta_Z ( t) \le 2 c_1 \kappa   e^{-c_2 t }.$$ 
Putting $ K := 2 c_1 \kappa $ and $ \theta = c_2, $ this implies the result, if we have already shown that $ \kappa $ is finite. This last fact follows  immediately from \eqref{eq:lyapunov}, following the first lines of the proof of Theorem 6.1 of \cite{MT1993}. Indeed, we have by Dynkin's formula that 
$$  e^{\alpha t } \E_z ( V ( Z_t) ) \le V(z)  + \frac{\beta}{\alpha} e^{\alpha t } ,$$
implying that 
$$ \E_z (V( Z_t)) \le e^{- \alpha t } V(z) + \frac{\beta}{\alpha}.$$ 
Integrating this last inequality with respect to $ \mu ( dz) $ implies the result. 
\end{proof}


\bibliography{BIB}

\end{document}